\documentclass[a4paper,10pt]{amsart}
\usepackage{amsmath,amssymb,amsthm,amscd}

\theoremstyle{plain}
\newtheorem{theorem}{Theorem}[section]
\newtheorem{proposition}[theorem]{Proposition}
\newtheorem{lemma}[theorem]{Lemma}
\newtheorem{corollary}[theorem]{Corollary}

\theoremstyle{definition}

\newtheorem{example}[theorem]{Example}
\newtheorem{remark}[theorem]{Remark}

\numberwithin{equation}{section}
\newcommand{\Q}{\mathbb{Q}}
\newcommand{\N}{\mathbb{N}}
\newcommand{\Z}{\mathbb{Z}}
\newcommand{\set}[2]{\{#1|\ #2\}}

\newcommand{\sub}{\subseteq}
\newcommand{\ord}{\mathrm{ord}}
\newcommand{\gen}[1]{\langle #1\rangle}

\newcommand{\T}{\mathcal{T}}
\newcommand{\U}{\mathcal{U}}

\begin{document}
\title{Notes on additively divisible commutative semirings}

\author[T.~Kepka]{Tom\'{a}\v{s}~Kepka}
\address{Charles University, Faculty of Mathematics and Physics, Department of Algebra \\
Sokolovsk\'{a} 83, 186 75 Prague 8, Czech Republic}
\email{kepka@karlin.mff.cuni.cz}

\author[M.~Korbel\'a\v r]{\textsc{Miroslav Korbel\'a\v r}}
\address{Department of Mathematics and Statistics, Faculty of Science, Masaryk University, Kotl\' a\v rsk\'{a} 2, 611 37 Brno, Czech Republic}
\email{miroslav.korbelar@gmail.com}

\thanks{This work is a part of the research project MSM00210839 financed
by M\v SMT. The first author was supported by the Grant Agency of Czech Republic, \#201/09/0296 and the second author by the project LC 505 of Eduard \v Cech's Center for Algebra and Geometry.}

\keywords{commutative semiring, divisible semigroup, idempotent, torsion}
\subjclass[2010]{16Y60, 20M14}


\begin{abstract}
Commutative semirings with divisible additive semigroup are studied. We show that an additively divisible commutative semiring is idempotent, provided that it is finitely generated and torsion. In case that a one-generated additively divisible semiring posseses no unit, it must contain an ideal of idempotent elements. We also present a series of open questions about finitely generated commutative semirings and their equivalent versions. 
\end{abstract}

\maketitle
\vspace{4ex}


It is well known that a commutative field is finite provided that it is a finitely generated ring. Consequently, no finitely generated commutative ring (whether unitary or not) contains a copy of the field $\Q$ of rational numbers. On the other hand, it seems to be an open problem whether a finitely generated (commutative) semiring $S$ can contain a copy of the semiring (parasemifield) $\Q^{+}$ of positive rationals. Anyway, if $S$ were such a (unitary) semiring with $1_{S}=1_{\Q^{+}}$, then the additive semigroup $S(+)$ should be divisible. So far, all known examples of finitely generated additively divisible commutative semirings are additively idempotent. Hence a natural question arises, whether a finitely generated (commutative) semiring with divisible additive part has to be additively idempotent.

 Analogous questions were studied for semigroups. According to \cite[2.5(iii)]{decomp}, there is a finitely generated non-commutative semigroup with a divisible element that is not idempotent. In this context it is of interest to ask whether there are non-finite but finitely generated divisible semigroups. This is in fact not true. Moreover, there exist infinite but finitely generated divisible (non-commutative) groups (see \cite{guba} and \cite{olshan}).

The present short note initiates a study of (finitely generated) additively divisible commutative semirings. In particular, We show that an additively divisible commutative semiring is idempotent, provided that it is finitely generated and torsion. The one-generated case is treated in more detail and we prove that in the case that a one-generated additively divisible semirings possess no unit, it must contain an ideal of idempotent elements. However, even the one-generated case, in the mentioned conjecture, remains an open problem.
Finally, we present a series of open questions about finitely generated commutative semirings and their equivalent versions.

\section{Preliminaries}

Throughout the paper, all algebraic structures involved (as semigroups, semirings, groups and rings) are assumed to be commutative, but, possibly, without additively and/or multiplicatively neutral elements. Consequently, a \emph{semiring} is a non-empty set equipped with two commutative and associative binary operations, an addition and a multiplication, such that the multiplication distributes over the addition. A semiring $S$ is called a ring if the additive semigroup of $S$ is a group and $S$ is called a \emph{parasemifield} if the multiplicative semigroup of $S$ is a non-trivial group.

We will use the usual notation: $\N$ for the semiring of positive integers, $\N_{0}$ for the semiring of non-negative integers, $\Q^{+}$ for the parasemifield of positive rationals and $\Q$ for the field of rationals.

For all $k,t\in\N$ define a relation $\rho(k,t)=\set{(m,n)\in \N\times\N}{m-n\in\Z t\ \&\ (m\neq n \Rightarrow m,n\geq k)}$. It is well known that the relations $\mathrm{id}_{\N}$ and $\rho(k,t)$, where $k,t\in\N$, are just all congruences of the additive semigroup $\N(+)$ (and of the semiring $\N$ as well).

Let $S$ be a semiring. An element $a\in S$ is said to be \emph{of finite order} if the cyclic subsemigroup $\N a=\set{ka}{k\in\N}$, generated by the element $a$, is finite. We put $\ord(a)=\mathrm{card}(\N a)$ in this case. If $\T(S)$ denotes the set of elements of finite order, than either $\T(S)=\emptyset$ or $\T(S)$ is a \emph{torsion} subsemiring (even an ideal) of $S$ (i.e., every element of $\T(S)$ has finite order).

\begin{lemma}\label{1.2}
Let $\emptyset\neq A$ be subset of a semiring such that $\sup\set{\ord(a)}{a\in A}<\infty$. Then $\sup\set{\ord(b)}{b\in \gen{A}}<\infty$. Moreover, there is $r\in\N$ such that $2rb=rb$ for every $b\in\gen{A}$.
\end{lemma}
\begin{proof}
Set $m=\sup\set{\ord(a)}{a\in A}\in\N$. We have $\N a\cong \N(+)/\rho(k_{a},t_{a})$, where $k_{a},t_{a}\in\N$, $k_{a}+t_{a}\leq m+1$. Furthermore, $2m_{a}a=m_{a}a$ for some $m_{a}\in\N$, $m_{a}\leq m+1$. Setting $r=(m+1)!$, we get $2rb=rb$ for every $b\in\gen{A}$. Of course, $\N b\cong \N(+)/\rho(k_{b},t_{b})$ and $\ord(b)=k_{b}+t_{b}-1$. Since $2rb=rb$, we have $r\geq k_{b}$ and $t_{b}$ divides $2r-r=r$. Consequently, $k_{b}+t_{b}-1\leq n$ for $n=2r-1\in\N$.
\end{proof}

\begin{lemma}\label{1.3}
Let $a,b\in S$ be such that $ka=la+b$ for some $k,l\in\N$, $k\neq l$. If $\ord(b)$ is finite, then $\ord(a)$ is so.
\end{lemma}
\begin{proof}
There are $m,n\in\N$ such that $m<n$ and $mb=nb$. Then $nka=nla+nb=nla+mb=(n-m)la+m(la+b)=(n-m)la+mka=((n-m)l+mk)a$. Since $k\neq l$, we see that $(n-m)k\neq (n-m)l$ and $nk\neq(n-m)l+mk$. Consequently, $\ord(a)$ is finite.
\end{proof}

On a semiring $S$ define a relation
 $\sigma_{S}=\set{(a,b)\in S\times S}{(\exists m\in\N)\ ma=mb}$. Clearly, $\sigma_{S}$ is a congruence of $S$ and $\sigma_{S/\sigma_{S}}=id$.

\begin{lemma}\label{1.6}
A semiring $S$ is torsion, provided that the factor-semiring
$S/\sigma_{S}$ is torsion.
\end{lemma}
\begin{proof}
For every $a\in S$ there are $k,l\in\N$ such that
$(ka,la)\in\sigma_{S}$, $k<l$. Furthermore, there is $m\in\N$
with $mka=mla$. Clearly, $mk<ml$, and hence $\ord(a)$ is
finite.
\end{proof}

Let $S$ be a semiring and $o\notin S$ be a new element. Putting $x+o=o+x=x$ and $oo=xo=ox=o$ we get again a semiring  $S\cup\{o\}$. Consider now the semiring $T=\N_{0}\times (S\cup\{o\})$ equipped with component-wise addition and multiplication given as  $(n,a)(m,b)=(nm,ma+nb+ab)$ for every $n,m\in\N_{0}$ and $a,b\in S\cup\{o\}$. Denote $\U(S)=T\setminus\{(0,o)\}$ the (unitary) subsemiring.

Now, $S$ can always be treated as a natural unitary $\U(S)$-semimodule, with $(n,a)x=nx+ax$ for every $(n,a)\in\U(S)$, $x\in S$. (Here $(n,o)x=nx$ and $(0,a)x=ax$ for $n\in\N$ and $a,x\in S$.)

\begin{lemma}\label{1.9}
Let $S$ be a semiring. If $w\in S$, $a,b,c\in\U(S)w$ and $m\in\N$ are such that $ma=mb$ and $mc=w$, then $a=b$.
\end{lemma}
\begin{proof}
For every $d\in\U(S)w$, there is $\alpha_{d}\in\U(S)$ with $d=\alpha_{d}w$. Now, $a=\alpha_{a}w=\alpha_{a}mc=\alpha_{c}\alpha_{a}mw=\alpha_{c}ma=\alpha_{c}mb=\alpha_{c}\alpha_{b}mw=\alpha_{b}mc=\alpha_{b}w=b$.
\end{proof}

\section{Additively divisible semirings}

Let $S(=S(+))$ be a semigroup. An element $a\in S$ is called \emph{divisible} (\emph{uniquely divisible}, resp.) if for every $n\in\N$ there exists  $b\in T$ (a unique, resp.) such that $a=nb$.

A semigroup $S$ is called \emph{divisible} (\emph{uniquely divisible}, resp.) if every element of $S$ is divisible (uniquely divisible, resp.). Clearly, $S$ is divisible iff $S=nS$ for every $n\in\N$. The class of divisible semigroups is closed under taking homomorphic images and cartesian products and  contains all divisible groups and all semilattices (i.e., idempotent semigroups).

A semiring is called \emph{additively divisible} (\emph{additively uniquely divisible}, resp.) if its additive part is a divisible semigroup (uniquely divisible semigroup, resp.). The semiring $\Q^{+}$ is additively (uniquely) divisible.

Note, that any semigroup $S(+)$ with an idempotent element $e$ (i.e. $e+e=e$) can always be treated as a semiring with a constant multiplication given as $ab=e$ for every $a,b\in S$.

The following theorem is a consequence of \cite[2.5(i)]{decomp}.

\begin{theorem}\label{2.1}\cite{decomp}
A semigroup $S$ is finitely generated and
divisible if and only if $S$ is a finite semilattice.
\end{theorem}

\begin{theorem}\label{2.3}
The following are equivalent for a  commutative semiring $S$:
\begin{enumerate}
\item[(i)] $S$ is  additively idempotent.
\item[(ii)] $S$ is additively divisible and bounded (i.e. $\sup\set{\ord(a)}{a\in S}<\infty$).
\item[(iii)] $S$ is additively uniquely divisible and torsion.
\end{enumerate}
\end{theorem}
\begin{proof}
(i)$\Rightarrow$(ii): Obvious.

(ii)$\Rightarrow$(i): By \ref{1.2} there exists $n\in\N$ such that
$2na=na$ for every $a\in S$. Since
$S$ is divisible, we have $a=nb$, and so $2a=2nb=nb=a$.

(i)$\Rightarrow$(iii): Easy.

(iii)$\Rightarrow$(i): Since $S$ is torsion, for every $a\in S$ there is $k\in\N$ such that $2ka=ka$. Now, $S$ is additively uniquely divisible, hence $k(2a)=ka$ and $2a=a$.
\end{proof}

\begin{corollary}\label{2.4}
Let $S$ be an divisible semiring. Then the factor-semiring $S/\sigma_{S}$ is additively uniquely divisible.

Moreover, if $S$ is a torsion, then $\sigma_{S}$ is just
the smallest congruence of $S$ such that the corresponding
factor-semiring is additively idempotent.
\end{corollary}

\begin{corollary}
  Let $S$ be an additively divisible semiring with $\T(S)\neq\emptyset$. Then $\T(S)$ is an additively divisible ideal in $S$.

If, moreover, $S$ is uniquely additively divisible, then  $\T(S)$ is additively idempotent.
\end{corollary}

\begin{example}\label{2.5}
Consider a non-trivial semilattice $S$ and the quasicyclic
$p$-group $\Z_{p^{\infty}}$. Then the product
$S\times\Z_{p^{\infty}}$ is a torsion divisible semigroup that is
neither a semilattice nor a group.
\end{example}

\section{Additively divisible semirings - continued}

Throughout this section, let $S$ be an additively divisible semiring. The following assertion is easy to verify.

\begin{proposition}\label{3.2}
Assume that the semiring $S$ is generated as an $S$-semimodule by
a subset $A$ such that $\sup\set{\ord(a)}{a\in A}<\infty$. Then $S$ is additively idempotent.
\end{proposition}
\begin{proof}
Set $m=\sup\set{\ord(a)}{a\in A}\in \N$ and put $B=\set{b\in S}{\ord(b)\leq m}$. Then $A\sub B$ and
$sb\in B$ for all $s\in S$ and $b\in B$. Consequently, $\gen{B}=S$
and, by \ref{1.2}, $\sup\set{\ord(b)}{b\in S}<\infty$. Now, it remains to use \ref{2.3}.
\end{proof}

\begin{corollary}\label{3.3}
The semiring $S$ is additively idempotent, provided that it is generated as an $S$-semimodule by a finite set of elements of finite orders.
\end{corollary}

\begin{corollary}\label{3.4}
The semiring $S$ is additively idempotent, provided that it is torsion and finitely generated.
\end{corollary}

\begin{remark}\label{3.5}
(i) The zero multiplication ring defined on $\Z_{p^{\infty}}$ is both
additively divisible and additively torsion. Of course, the ring
is neither additively idempotent nor finitely generated. The
(semi)group $\Z_{p^{\infty}}(+)$ is not uniquely divisible.

(ii) Let $R$ be a (non-zero) finitely generated ring (not necessary with unit). Then $R$ has at
least one maximal ideal $I$ and the factor-ring $R/I$ is a finitely
generated simple ring. However, any such a ring is finite and
consequently, $R$ is not additively divisible.
\end{remark}

\begin{remark}\label{3.6}
Assume that $S$ is additively uniquely divisible.

(i) For all $m,n\in\N$ and $a\in S$, there is a uniquely
determined $b\in S$ such that $ma=nb$ and we put $(m/n)a=b$. If
$m_{1},n_{1}\in\N$ and $b_{1}\in S$ are such that
$m/n=m_{1}/n_{1}$ and $m_{1}a=n_{1}b_{1}$, then $k=mn_{1}=m_{1}n$
and $kb_{1}=mm_{1}a=kb$ and $b_{1}=b$. Consequently, we get a
(scalar) multiplication $\Q^{+}\times S\to S$ (one checks easily
that $q(a_{1}+a_{2})=qa_{1}+qa_{2}$,
$(q_{1}+q_{2})a=q_{1}a+q_{2}a$, $q_{1}(q_{2}a)=(q_{1}q_{2})a$ and
$1a=a$ for all $q_{1},q_{2}\in\Q^{+}$ and $a_{1},a_{2},a\in S$)
and $S$ becomes a unitary $\Q^{+}$-semimodule. Furthermore,
$qa_{1}\cdot a_{2}=a_{1}\cdot qa_{2}$ for all $q\in\Q^{+}$ and
$a_{1},a_{2}\in S$, and therefore $S$ is a unitary
$\Q^{+}$-algebra.

(ii) Let $a\in S$ be multiplicatively but not additively
idempotent (i.e., $a^{2}=a\neq 2a$). Put $Q=\Q^{+}a$. Then $Q$ is
a subalgebra of the $\Q^{+}$-algebra $S$ and the mapping
$\varphi:q\mapsto qa$ is a homomorphism of the $\Q^{+}$-algebras
and, of course, of the semirings as well. Since $a\neq 2a$, we
have $\ker(\varphi)\neq\Q^{+}\times\Q^{+}$. But $\Q^{+}$ is a
congruence-simple semiring and it follows that $\ker(\varphi)=id$.
Consequently, $Q\cong\Q^{+}$.

Put $T=Sa$. Then $T$ is an ideal of the $\Q^{+}$-algebra $S$,
$Q\sub T$ (we have $qa=qa\cdot a\in T$) and $a=1_{Q}=1_{T}$ is a
multiplicatively neutral element of $T$. The mapping $s\mapsto sa$, $s\in S$,
is a homomorphism of the $\Q^{+}$-algebras. Consequently, $T$ is
additively uniquely divisible. Furthermore, $T$ is a finitely
generated semiring, provided that $S$ is so.
\end{remark}

\begin{proposition}\label{3.7}
Assume that $1_{S}\in S$. Then:
\begin{enumerate}
\item[(i)] $S$ is additively uniquely divisible.
\item[(ii)] Either $S$ is additively idempotent or $S$ contains a
subsemiring $Q$ such that $Q\cong\Q^{+}$ and $1_{S}=1_{Q}$.
\item[(iii)] If $\ord(1_{S})$ is finite, then $S$ is
additively idempotent.
\end{enumerate}
\end{proposition}
\begin{proof}
For every $m\in\N$, there is $s_{m}\in S$ such that
$1_{S}=ms_{m}$. That is, $s_{m}=(m1_{S})^{-1}$. If $ma=mb$, then
$a=s_{m}ma=s_{m}mb=b$ and we see that $S$ is additively uniquely
divisible. The rest is clear from \ref{3.6}.
\end{proof}

\begin{proposition}\label{3.8}
If the semiring  $S$ is non-trivial and additively cancellative, then $S$ is not finitely generated.
\end{proposition}
\begin{proof}
The difference ring $R=S-S$ of $S$ is additively divisible, and hence it is not finitely generated by \ref{3.5}(ii). Then $S$ is not finitely generated either.
\end{proof}

\section{One-generated additively divisible semirings}

In this section, let $S$ be an additively divisible semiring generated by a single element $w\in S$.

\begin{proposition}\label{4.1}
The semiring $S$ is additively uniquely divisible.
\end{proposition}
\begin{proof}
Follows from \ref{1.9}.
\end{proof}

\begin{proposition}\label{4.2}
The semiring $S$ is additively idempotent, provided that $\ord(w^{m})$ is finite for some $m\in\N$.
\end{proposition}
\begin{proof}
Let $n\in\N$ be the smallest number with $\ord(w^{n})$ finite. If $n=1$, then the result follows from \ref{3.3}, and so we assume, for contrary, that $n\geq 2$. Since $S(+)$ is divisible, there are $v\in S$, $k\in\N_{0}$ and a polynomial $f(x)\in\N_{0}[x]\cdot x$ such that $w=2v$ and $v=kw+w^{2}f(w)$. Hence $w^{n-1}=2kw^{n-1}+2w^{n}f(w)$. By assumption, $2w^{n}f(w)$ is of finite order. If $k=0$, then clearly $\ord(w^{n-1})$ is finite, and if $k\geq 1$, then $\ord(w^{n-1})$ is finite, by \ref{1.3}, the final contradiction.
\end{proof}

For $n\in\N$ denote $S^{n}=\gen{\set{a_{1}\dots a_{n}}{a_{i}\in S}}$.

\begin{lemma}\label{4.3.0}
 If $u\in S$ is such that $w=wu$, then $u=1_{S}$.
\end{lemma}
\begin{proof}
Let $w=wu$. Since $S=\U(S)w$,  for every $a\in S$ there is $\alpha_{a}\in\U(S)$ such that $a=\alpha_{a}w$. Hence $a=\alpha_{a}w=\alpha_{a}wu=au$ for every $a\in S$. Thus $u=1_{S}$.
\end{proof}

\begin{corollary}\label{4.3}
 $1_{S}\in S$ (i.e. $S$ is unitary) if and only if $S^{2}=S$. In this case:
\begin{enumerate}
\item[(i)] There exists $w^{-1}$ in $S$.
\item[(ii)] $S^{n}=S^{m}$ for all $n,m\in\N$.
\end{enumerate}
\end{corollary}
\begin{proof}
Let $S^{2}=S$. Then $w\in S=S^{2}$ and there is a non-zero polynomial $f(x)\in\N_{0}[x]\cdot x$ such that $w=wf(w)$. By \ref{4.3.0}, $f(w)=1_{S}$. Now, since $1_{S}\in S=S^{2}$, we have similarly that $1_{S}=wu$ for some $u\in S$. Hence $w^{-1}=u\in S$. The rest is obvious.
\end{proof}

\begin{lemma}\label{4.4}
Let $k\in\N$, $k\geq 2$ and $u\in S\cup\{0\}$ be such that $w=kw+u$. Then there exists $a\in S\cup\{0\}$ such that  $x=2x+ax$ for every $x\in S$.
\end{lemma}
\begin{proof}
Let $u=mw+wf(w)$, where $m\in\N_{0}$ and $f(x)\in\N_{0}[x]\cdot x$. If $f(x)=0$, then $\ord(w)$ is finite, $S$ is idempotent by \ref{3.3} and we can put $a=0$.

 Hence assume that $f(x)\neq 0$. Put $n=m+k$ and $b=f(w)\in S$. We have $w=nw+wb$. Adding $(n-2)w$ to both sides of this equality, we get $(n-1)w=2(n-1)w+wb$. Since $w$ is a generator and $S$ is additively divisible we have $b=(n-1)a$ for some $a\in S$ and for every $x\in S$ there is $\alpha_{x}\in\U(S)$ such that $x=(n-1)\alpha_{x}w$. Hence $x=\alpha_{x}(n-1)w=\alpha_{x}(2(n-1)w+(n-1)wa)=2x+xa$ for every $x\in S$.
\end{proof}

\begin{remark}\label{rem_4}
Let $x,v\in S$ be such that $x=2x+v$. Put $o=x+v$. Then:
\begin{enumerate}
 \item[(i)] $o$ is an idempotent and $x=x+o$.
 \item[(ii)] The set $\set{b\in S}{x=x+b}$ is a subsemigroup of $S(+)$ with $o$ as an absorbing element.
\end{enumerate}
\end{remark}

\begin{theorem}
If  $1_{S}\notin S$, then there exists $a\in S\cup\{0\}$ such that  $\set{u+au}{u\in S}$ is an ideal in $S$ consisting of idempotent elements.
\end{theorem}
\begin{proof}
Since $S$ is divisible, there is $f\in\N_{0}[x]$ such that $w=2f(w)$. By \ref{4.3.0}, we get that $f(x)=kx+xg(x)$ for some $k\in\N$, $k\geq 2$ and $g\in\N_{0}[x]$. Now, by \ref{4.4} there is  $a\in S\cup\{0\}$ such that  $u=2u+au$ for every $u\in S$. Finally, $u+au$ is idempotent for every $u\in S$ by \ref{rem_4}.
\end{proof}

Clearly, $S$ is additively divisible, iff for every $n\in\N$ there is a non-zero polynomial $f_{n}(x)\in\N_{0}[x]\cdot x$ such  that $w=nf_{n}(w)$. By \cite[2.5(i)]{decomp}, every divisible element in a finitely generated commutative semigroup has to be idempotent. Assuming now, that the degree of the polynomials $f_{n}$ is bounded, we get that $S$ is additively idempotent. To illustrate some other situations and techniques see the next examples.

\begin{example}
Let $n,m,k,l\in\N$ be such that $n^{l-1}\neq m^{k-1}$ and suppose that $w=nw^{k}$ and $w=mw^{l}$. Then $S$ is additively idempotent.

Indeed, $nm^{k}w^{kl}=n(mw^{l})^{k}=nw^{k}=w=mw^{l}=m(nw^{k})^{l}=mn^{l}w^{kl}$. Since $nm^{k}\neq mn^{l}$, $w^{kl}$ is of finite order and $S$ is additively idempotent by \ref{4.2}.
\end{example}

\begin{example}
Let $w=2(w+w^{2})$ and $w=3(w+w^{3})$. Then $S$ is additively idempotent.

Indeed, adding these two equalities we get $2w=5w+2w^{2}+3w^{3}$. Now, since $w^{2}=2w^{2}+2w^{3}$ we can substitute $2w=5w+(2w^{2}+2w^{3})+w^{3}=5w+w^{2}+w^{3}$. Hence $4w=10w+(2w^{2}+2w^{3})$ and by the same substitution we get $4w=10w+w^{2}$. Finally, $8w=20w+2w^{2}=18w+(2w+2w^{2})=19w$ and $w$ is of finite order. Thus $S$ is additively idempotent by \ref{4.2}.
\end{example}

Even when a one-generated semiring does not have a unit, we can use another one that will contain a unit and will share some properties with the original semiring (see \ref{unit_adding}).

\begin{proposition}\label{unit_adding}
 $T_{S}=\set{\varphi\in\mathrm{End}(S(+))}{(\exists \alpha\in\U(S))(\forall x\in S)\ \varphi(x)=\alpha x}$ is an additively divisible commutative semiring that has a unit $\mathrm{id}_{S}$ and is generated by the set $\{\mathrm{id}_{S},\varphi_{w}\}$, where $\varphi_{w}(x)=wx$.

Moreover, $S$ is additively idempotent if and only if $T_{S}$ is so.
\end{proposition}
\begin{proof}
We only need to show, that $\mathrm{id}_{S}\in T_{S}$ is an additively divisible element. Let $n\in\N$. Then $w=na$ for some $a\in S$ and  $a=\alpha_{a}w$ for some $\alpha_{a}\in\U(S)$. Now, for every $x\in S$ there is $\alpha_{x}\in\U(S)$ such that $x=\alpha_{x}w$. Hence $x=\alpha_{x}w=\alpha_{x}n\alpha_{a}w=n\alpha_{a}x$. Thus $\mathrm{id}_{S}=n\alpha_{a}$. The rest is easy.
\end{proof}

\section{A few conjectures}

In this last section we present some other open questions that are influenced by our main problem (namely by the conjecture (A) - see below).

\begin{example}\label{8.0}
Given a multiplicative abelian group $G$ and an element $o\notin G$, put $U(G)=G\cup\{o\}$ and define addition and multiplication on $U(G)$ (extending the multiplication on $G$) by $x+y=xo=ox=o$ for all $x,y\in U(G)$. Then $U(G)$ becomes an ideal-simple semiring.
\end{example}

Consider the following statements:

\begin{enumerate}
\item[(A)] Every finitely generated additively divisible semiring is additively idempotent.
\item[(A1)] Every finitely generated additively uniquely divisible semiring is additively idempotent.
\item[(B)] No finitely generated semiring contains a copy of $\Q^{+}$.
\item[(B1)] No finitely generated semiring with a unit element contains a copy of $\Q^{+}$ sharing the unit.
\item[(B2)] Every finitely generated additively divisible semiring with a unit is additively idempotent.
\item[(C)] Every parasemifield, that is finitely generated as a semiring, is additively idempotent.
\item[(C1)] Every finitely generated infinite and ideal-simple semiring is additively idempotent or a copy of the semiring $U(G)$ (see \ref{8.0}) for an infinite finitely generated abelian group $G$.

\end{enumerate}

\begin{proposition}\label{8.1}
$(\mathrm{A})\Leftrightarrow(\mathrm{A1})\Rightarrow(\mathrm{B})\Leftrightarrow(\mathrm{B1})\Leftrightarrow(\mathrm{B2})\Rightarrow(\mathrm{C})\Leftrightarrow(\mathrm{C1})$.
\end{proposition}
\begin{proof}
First, it is clear that $(\mathrm{A})\Rightarrow(\mathrm{A1})$, $(\mathrm{B})\Rightarrow(\mathrm{B1})$  and  $(\mathrm{A})\Rightarrow(\mathrm{B2})$. Furthermore, $(\mathrm{C})\Leftrightarrow(\mathrm{C1})$ by \cite[5.1]{a note}.
Now, assume that (A1) is true and let $S$ be a finitely generated additively divisible semiring. By \ref{2.4}, $S/\sigma_{S}$ is additively uniquely divisible and, of course, this semiring inherits the property of being finitely generated. By (A1), the semiring $S/\sigma_{S}$ is additively idempotent, and hence the semiring $S$ is additively torsion by \ref{1.6}. Finally, $S$ is additively idempotent by \ref{3.4}. We have shown that $(\mathrm{A1})\Rightarrow(\mathrm{A})$ and consequently, $(\mathrm{A})\Leftrightarrow(\mathrm{A1})$.

Next, let (B1) be true and let $S$ be a finitely generated semiring containing a subsemiring $Q\cong\Q^{+}$. Put $P=S\cdot 1_{Q}$. Then $P$ is an ideal of $S$, $1_{Q}=1_{P}$, $Q\sub P$ and the map $s\mapsto s1_{Q}$ is a homomorphism of $S$ onto $P$. Thus $P$ is a finitely generated semiring and this is a contradiction with (B1). We have shown that $(\mathrm{B1})\Rightarrow(\mathrm{B})$ and consequently, $(\mathrm{B})\Leftrightarrow(\mathrm{B1})$.

The implication $(\mathrm{B2})\Rightarrow(\mathrm{B1})$ is easy, since every semiring $S$ with a unit element $1_{S}$ that contains a subsemiring $Q\cong\Q^{+}$ with $1_{S}=1_{\Q^{+}}$ is additively divisible (for $a\in S$ and $m\in\N$ choose $b=(m1_{S})^{-1}a\in S$ and get $a=mb$).

The implication $(\mathrm{B})\Rightarrow(\mathrm{B2})$ follows immediately from \ref{3.7}.

We have shown that $(\mathrm{B})\Rightarrow(\mathrm{B2})\Rightarrow(\mathrm{B1})\Leftrightarrow(\mathrm{B})$. 

Finally, the implication $(\mathrm{B2})\Rightarrow(\mathrm{C})$ follows from the fact that every parasemifield $S$ is additively divisible, since the prime parasemifield of $S$ containg $1_{S}$ is isomorhic either to $\Q^{+}$ or to the trivial semiring.
\end{proof}

The conjecture (C) was confirmed for the one-generated case in \cite{notes} and for the two-generated case in \cite{jezek}.

\

Note that using the Birkhoff's theorem we can consider an equivalent version of the conjecture (A):

\begin{enumerate}
\item[(A')] Every finitely generated subdirectly irreducible additively divisible semiring is additively idempotent.
\end{enumerate}

Of course, it would be sufficient if such a semiring were finite. Unfortunately, this is not true. Assume for instance the semiring $S=V(G)$ with $G=\Z(+)$. By \cite[10.1]{simple}, $S$ is simple and two-generated. Nevertheless, it is an open question whether also one-generated subdirectly irreducible additively divisible semiring can be infinite.

Finally, Mal'cev \cite{malcev} proved that every finitely generated commutative semigroup is residually finite (i.e. it is a subdirect product of finite semigroups). Notice, that also the additive part of a free finitely generated additively idempotent semiring is a residually finite semigroup.  If this is true also for every finitely generated additively divisible semiring, we get a nice positive answer to the conjecture (A).

\bigskip



\end{document}